\documentclass[12pt]{article}
\usepackage{amsmath,amsfonts,amssymb,amsthm,amscd}
\title{Connected components of definable groups and $o$-minimality I}
\date{January 27th, 2011}
\author{Annalisa Conversano\\University of Konstanz \and Anand Pillay\thanks{Supported by EPSRC grant EP/I002294/1}\\University of Leeds}
\newtheorem{Theorem}{Theorem}[section]
\newtheorem{Proposition}[Theorem]{Proposition}
\newtheorem{Definition}[Theorem]{Definition} 
\newtheorem{Remark}[Theorem]{Remark}
\newtheorem{Lemma}[Theorem]{Lemma}
\newtheorem{Corollary}[Theorem]{Corollary}
\newtheorem{Fact}[Theorem]{Fact}
\newtheorem{Example}[Theorem]{Example}

\newtheorem{Problem}[Theorem]{Problem}
\newtheorem{Conjecture}[Theorem]{Conjecture}
\newcommand{\R}{\mathbb R}   
  
\newcommand{\Z}{\mathbb Z}

\begin{document}
\maketitle

\begin{abstract} 
We give examples of definable groups $G$ (in a saturated model, sometimes $o$-minimal) such that 
$G^{00} \neq G^{000}$,  yielding also new examples of ``non $G$-compact" theories. We also prove that for $G$ definable in a (saturated) $o$-minimal structure, $G$ has a ``bounded orbit" (i.e. there is a type of $G$ whose stabilizer has bounded index) if and only if $G$ is definably amenable, giving a positive answer to a conjecture of Newelski and Petrykowski in this {\em special case} of groups definable in $o$-minimal structures. 
We also introduce and discuss further conjectures on bounded orbits in definable groups. These results and analyses are informed by a decomposition theorem for groups in $o$-minimal structures.
\end{abstract}

\section{Introduction and preliminaries}
In this paper groups definable in $o$-minimal and closely related structures are studied, partly for their own sake and partly as a ``testing ground" for general conjectures. 
Given a $\emptyset$-definable group $G$ in a saturated structure ${\bar M}$, $G^{00}_{\emptyset}$ is the smallest subgroup of $G$ of bounded index which is type-definable over $\emptyset$, and $G^{000}_{\emptyset}$ is the smallest subgroup of $G$ of bounded index which is $Aut({\bar M})$-invariant. In $o$-minimal structures and more generally theories with $NIP$, these ``connected components" remain unchanged after naming parameters and so are just referred to as $G^{00}$ and $G^{000}$. In any case $G^{00}_{\emptyset}$ and $G^{000}_{\emptyset}$ are ``definable group" analogues of the groups of $KP$-strong automorphisms and Lascar strong automorphisms, respectively, of a saturated structure. The relationship between these definable group and automorphism group notions is explored in \cite{Gismatullin-Newelski}. Although examples were given in \cite{CLPZ} where the strong automorphism groups differ, until now no example was known where $G^{000}_{\emptyset} \neq G^{00}_{\emptyset}$.  In this paper (Section 3) we give a ``natural" example: $G$ is simply a saturated elementary extension of $\widetilde{SL(2,\R)}$ (the universal cover of $SL(2,\R)$) in the language of groups. $G$ is {\em not} actually definable in an $o$-minimal structure,  but we give another closely related example which is. In any case the two-sorted structure consisting of $G$ and a principal homogeneous space for $G$ is now a (natural) example of a ``non $G$-compact" structure (or theory) i.e. where the group of Lascar strong automorphisms  is a properly contained in the group of $KP$-strong automorphisms. 

Another fruitful theme in recent years has been the generalization of stable group theory outside the stable context. The $o$-minimal case has been important and there is now a good understanding of ``definably compact" groups from this point of view; for example they are definably amenable, ``generically stable for measure", and $G$ is dominated by $G/G^{00}$. In the current paper we try to go beyond the definably compact setting, motivated partly by questions of Newelski and Petrykowski. In \cite{NIPI}, definable groups $G$  with  ``finitely satisfiable generics" (which include definably compact groups in $o$-minimal structures)  were shown to be definably amenable by lifting the Haar measure on $G/G^{00}$ to a left invariant Keisler measure on $G$, making use of a global {\em generic type} $p$, whose stabilizer is $G^{00}$. We guess this encouraged Petrykowski to suggest that if a definable group $G$ (in any structure) has a global type whose stabilizer has ``bounded index" then $G$ is definably amenable. In Section 4 we confirm this conjecture when $G$ is definable in an $o$-minimal structure, as well as raise questions about nature of types with bounded orbit in the $o$-minimal and more generally $NIP$ environment.

In Section 2 of the paper we give a rather basic decomposition theorem (implicit in the literature) for groups in 
$o$-minimal structures, which is  useful for understanding the issues around definable amenability and bounded orbits, as well as $G^{00}$ and $G^{000}$ (although Section 3 can be more or less read independently of Section 2). We introduce and discuss the notion of $G$ having a ``good decomposition" (Definition 2.7), and in fact the $o$-minimal examples where $G^{00} \neq G^{000}$ will be also examples where good decomposition fails, although good decomposition does hold for algebraic groups. 

In a sequel to the current paper, \cite{Conversano-PillayII}, we will give a systematic account of $G^{00}$, $G^{000}$ as well as the quotient $G^{00}/G^{000}$, for groups $G$ in $o$-minimal structures. The decomposition theorem (2.6) as well as refinements of it will play a major role. 

%In an earlier version of the current paper, written just by the second author and circulated in December 2010, it was erroneously claimed that any definable group (in an $o$-minimal structure) has a good decomposition.  The first author pointed out a mistake in the proof, and this led to the collaboration on the current paper and its sequel. 

In general $T$ will denote a complete theory, $M$ an arbitrary model of $T$, and $G$ a group definable in $M$. 
We sometimes work in a sufficiently saturated and homogeneous model ${\bar M}$ of $T$, in which case ``small" or ``bounded" essentially means of cardinality strictly less than the degree of saturation of ${\bar M}$, but we will make the meaning more precise later in the paper. 
{\em Definability} usually means with parameters, and we say $A$-definable to mean definable with parameters from $A$ for $A$ a subset of $M$.
When we talk about $o$-minimal theories we will mean $o$-minimal expansions of the theory of real closed fields (and we leave it for later or to others to consider more general $o$-minimal contexts). In the $o$-minimal context, the important notion of definable compactness
was introduced by Peterzil and Steinhorn in \cite{PS}. For $X$ a definable subset of 
$M^{n}$, definable compactness of $X$ amounts to $X$ being closed and bounded in $M^{n}$. In the more general case of $X$ being a {\em 
definable manifold}, it means that for any definable function $f$ from $[0,1)$ to $X$, $lim_{x \to 1}f(x)$ exists in $X$.
When $G$ is a definable group, $G$ can be equipped with a definable manifold structure such that multiplication and inversion are continuous \cite{Pillay - groups}. Definable compactness of a definable group $G$ is then meant with respect to this definable manifold structure. But, as we are working in an $o$-minimal expansion of a real closed field, any definable group manifold $G$ can be assumed to be a definable subset of some $M^{n}$, and so definable compactness of $G$ reduces to $G$ being closed and bounded. 
{\em Definable connectedness} of $G$ is meant with respect to its definable manifold structure mentioned above. But it turns out that $G$ is definably connected in this sense if and only if $G$ has no proper definable subgroup of finite index (i.e. $G = G^{0}$). Any definable group $G$ is definably connected by finite, and so (in this $o$-minimal context) we will often assume that our definable groups are definably connected.
We will often use the well-known fact that any definably compact, definably connected, solvable normal definable subgroup $N$ of a definably connected group is central. This follows from Corollaries 5.3 and 5.4 of  \cite{Peterzil-Starchenko}. We will also use the fact that if $N$ is normal and definable in $G$, then $G$ is definably compact if and only if $N$ and $G/N$ are definably compact. (Following from \cite{NIPI} and \cite{NIPII}.)

In Section 4 of this paper we will make some references to ``stability-type" notions, $NIP$ theories, forking, etc. We generally refer the reader to \cite{NIPII} for the definitions, but make a few comments here. For ${\bar M}$ a saturated model of arbitrary theory $T$ and $G$ a group definable in ${\bar M}$, recall that $S_{G}({\bar M})$ denotes the space of complete types $p(x)$ over ${\bar M}$ such that $``x\in G" \in p$. $G$ (namely $G({\bar M})$) acts on $S_{G}({\bar M})$ on the left by  $gp = tp(ga/{\bar M})$ where  $a$ realizes $p$ in a bigger model. Slightly modifying Definition 5.1 from \cite{NIPII}, we will say that $p(x)\in S_{G}({\bar M})$ is {\em left $f$-generic}  if there is a small model $M_{0}$ such that for any $g\in G({\bar M})$, $gp$ does not fork over $M_{0}$.

The second author was partly motivated by some e-mail discussions with Hrushovski and Newelski in the late summer of 2010. Thanks to both of them for the inspiration, and in particular to Hrushovski for allowing us to include (in Section 4) some observations that he made on definable amenability. 

Many of the themes and results of this paper and the sequel appear in one form or another in the first author's doctoral thesis \cite{Conversano-thesis}, which is devoted to structural properties of groups definable in $o$-minimal structures (but does not explicitly discuss $G^{000}$). 
In particular the $o$-minimal example where $G^{00} \neq G^{000}$  (Example 2.10/Theorem 3.3) appears in her thesis as an example of a definable group {\em without} a definable ``Levi decomposition". 
In any case the first author would like to thank her advisor Alessandro Berarducci, as well as Ya'acov Peterzil for useful conversations.

%%%%%%%%%%%%%%%%%%%%%%%%%%%%%%%%%%%%%%%%%%%%%%%%%%%%%%%%%%%%%%%

\section{Decomposition theorems}
In this section $T$ is a complete $o$-minimal expansion of $RCF$, and we work in a model $M$ of $T$. $G$ will typically denote a definable, definably connected group, although we usually explicitly state definable connectedness. $K$ will denote the underlying real closed field of $M$. 
We first aim towards a useful ``basic decomposition theorem", Proposition 2.6 below (which is easily extracted from results in the literature).

%NEW
We begin by pointing out the existence, in every definable group, of a (unique) maximal normal definable torsion-free subgroup. As usual, for a positive integer $n$, an $n$-torsion element of $G$ is an element $x \in G$ such that $x^n = 1$, $1$ being the identity of the group (note that we are not assuming $G$ is commutative). We make use of results from \cite{Strzebonski}
connecting the existence of $n$-torsion elements with the $o$-minimal Euler characteristic of $G$.
Recall that if $\mathcal{P}$ is a cell decomposition of a definable set $X$, then the 
{\em $o$-minimal Euler characteristic} $E(X)$ is the number of even-dimensional cells in $\mathcal{P}$ minus the number of odd-dimensional cells in $\mathcal{P}$. This does not depend on $\mathcal{P}$, and when $X$ is finite then $E(X) = |X|$. A definable torsion-free group will be definably connected (Corollary 2.4 of \cite{PeSta}  but also follows from the proof of (ii) below). 

\begin{Proposition} \label{unifree}
(i) $G$ is torsion-free if and only if $G$ is ``solvable with no definably compact parts" in the sense of \cite{Edmundo}, namely there are
definable subgroups  $\{1\} = G_{0} < ... < G_{n} = G$ of $G$ such that for each $i<n$, $G_{i}$ is normal in $G_{i+1}$ and $G_{i+1}/G_{i}$ is $1$-dimensional and torsion-free. (In particular a torsion-free definable group is solvable.)
\newline
(ii) In every definable group $G$ there is a normal definable torsion-free subgroup which contains every
normal definable torsion-free subgroup of $G$. It is the unique normal definable torsion-free subgroup of $G$
of maximal dimension. We will refer to it as {\em the maximal normal definable torsion-free subgroup} of $G$, and note that it is invariant under all automorphisms of $(G,\cdot)$ which are definable in the ambient structure.
\end{Proposition}
\begin{proof}
(i) Right to left is obvious. Left to right follows (using induction) from Corollary 2.12 of \cite{PeSta} which states that if $G$ is torsion-free (and nontrivial) then there is a normal definable subgroup $H$ of $G$ such that $G/H$ is $1$-dimensional and torsion-free.
\newline
(ii) 
We recall that for definable groups $K < G$, 
\[
E(K) E(G/K) = E(G),
\]

and $G$ is torsion-free if and only if $E(G) = \pm 1$ (\cite{Strzebonski}). It follows that a quotient of torsion-free definable groups is still torsion-free (and hence torsion-free definable groups are definably connected).

Let $N$ be a normal definable torsion-free subgroup of $G$ of maximal dimension, and $H$ any normal definable torsion-free subgroup of $G$. We want to show that $H \subseteq N$.

We claim that $HN$ is a normal definable torsion-free subgroup of $G$: the definable group $H/(H \cap N)$ is torsion-free and it is definably isomorphic to $HN/N$. Thus   
$E(HN) = E(N)E(HN/N) = \pm1$ and $HN$ is torsion-free.  

But $N$ is of maximal dimension among the normal definable torsion-free subgroups of $G$, so 
$\dim(HN) = \dim(N)$. 
Since definable torsion-free groups are definably connected, it follows that $HN = N$, $H \subseteq N$ and 
$\dim H < \dim N$, unless $H = N$.
\end{proof}

%We first state the definitions and results from \cite{Edmundo} in a form suitable for our requirements.
%The first paraphrases Definition 1.4 of \cite{Edmundo}.
%\begin{Definition}  Let $G$ be a definable, solvable, definably connected group. $G$ is said to {\em have no definably compact parts}, if there is a chain  $1 = H_{0} < H_{1} < ... < H_{n} = G$ of definable, definably connected subgroups of $G$ such that for each $i=0,..,n-1$, $H_{i}$ is normal in $H_{i+1}$ and $H_{i+1}/H_{i}$ is one-dimensional, torsion-free (and divisible).
%\end{Definition}

Bearing in mind Proposition 2.1, the following proposition is easily deduced from Theorem 5.8 of \cite{Edmundo}, together with the fact that definably compact, definably connected, solvable definable groups are commutative:
\begin{Proposition} Let $G$ be a definable, solvable, definably connected group, and let $W$ be its maximal normal definable torsion-free subgroup. Then $G/W$ is definably compact and commutative. 
\end{Proposition}

\vspace{2mm}
\noindent
Recall that a definable group $G$ is said to be {\em semisimple} if $G$ has no definable, normal, definably connected, solvable (or commutative), nontrivial subgroups. Then, clearly, for an arbitrary definable group $G$, we have the exact sequence
$$ 1 \to R \to G \to G/R \to 1 $$

\noindent
where $R$, the {\em solvable radical} of $G$ is the maximal definable, normal, solvable, definably connected subgroup of $G$, and $G/R$ is semisimple.
If $R$ is definably compact then it is central in $G$.

\begin{Definition} We call a definable group $G$, {\em definably almost simple}, if $G$ is noncommutative and has no infinite (equivalently nontrivial, definably connected) proper definable normal subgroup. 
\end{Definition}

Note that if $G$ is definably almost simple, then $Z(G)$ is finite and $G/Z(G)$ is definable simple, and moreover $G$ is definably compact if and only if $G/Z(G)$ is definably compact.

\begin{Lemma} Let the definable group be semisimple and definably connected. Then there are definable, definably almost simple subgroups $H_{1},..,H_{t}$ of $G$ such that $G$ is the almost direct product of the $H_{i}$, namely there is a definable surjective homomomorphism from $H_{1}\times ... \times H_{t}$ to $G$ with finite kernel.
\end{Lemma}
\begin{proof} Well known. By \cite{PPSI}, $G/Z(G)$ is the direct product of definably simple groups  $B_{1},..,B_{t}$. Let $H_{i}$ be the definably connected component of the preimage of $B_{i}$ under the quotient map $G\to G/Z(G)$. 
\end{proof} 

\begin{Definition} Let $G$ be semisimple and definably connected. We say that $G$ has no definably compact part if in Lemma 2.4, no $H_{i}$ is definably compact. 
\end{Definition}

\noindent
We can now observe:

\begin{Proposition} Let $G$ be a definable (definably connected) group. Then there is a definable, definably connected normal, subgroup $W$ of $G$, and a definable, definably connected normal subgroup $C$ of $G/W$, such that
\newline
(i)  $W$ is torsion-free,
\newline
(ii)  $C$ is definably compact, and
\newline
(iii)  $(G/W)/C$ is semisimple with no definably compact part. 
\newline
$W$ is the maximal normal definable torsion-free subgroup of $G$, and $C$ is the maximal normal definable, definably compact, definably connected subgroup of $G/W$.
\end{Proposition}
\begin{proof} Let $R$ be the solvable radical of  $G$, and let $W$ be the maximal normal definable torsion-free subgroup of $R$ (given by Proposition 2.1). So $R/W$ is definably compact and commutative by 2.2. But let us note for now that since any definable torsion-free group is definably connected and solvable (\cite[2.11]{PeSta}), then $W$ coincides with the maximal normal definable torsion-free subgroup of $G$.

Now $R/W$ is the solvable radical of $G/W$ (and is also connected, definably compact, so in fact central in $G/W$), and $G/R$ is semisimple. Let us denote $G/R$ by  $H$ for now, and $\pi$ the surjective homomorphism from $G/W$ to $H$. Let $H_{1},..,H_{t}$ be given for $H$ by Lemma 2.4, namely the $H_{i}$ are definable, definably almost simple and $H$ is their (almost direct) product. Let $C_{1}$ be the product of those $H_{i}$ which are definably compact, and $D_{1}$ the product of the rest.  So $G/R =  H$ is the almost direct product of the semisimple definable groups $C_{1}$ and $D_{1}$. Let $C = (\pi^{-1}(C_{1}))$. So $C$ is an extension of  the definably compact connected group $C_{1}$ by the definably compact definably connected group $R/W$, hence is also definably compact and definably connected. Note that $C$ is normal in $G/W$, and the quotient $(G/W)/C$ is an image of $D_{1}$ (with finite kernel) so is semisimple with no definably compact parts.
\end{proof}

Let us fix notation for the data obtained in the proof  above, so as to be able to refer to them in the future. $R$ denotes the solvable radical of $G$ and $W$ the maximal normal definable torsion-free subgroup of $G$ (equivalently of $R$).
\newline
$G/R$ is the semisimple part of $G$ which can be written uniquely as $C_{1}\cdot D_{1}$ (almost direct product) where $C_{1}$ is semisimple and definably compact and $D_{1}$ is semisimple with no definably compact parts (and everybody is definably connected). 
\newline
We have the exact sequence
$$1 \to R/W \to G/W \to_{\pi} G/R = C_{1}\cdot D_{1}  \to 1$$ 
 and $C$ denotes $\pi^{-1}(C_{1})$ which is the maximal normal definable, definably connected, definably compact
 subgroup of $G/W$, and we call it the {\em normal definably compact part} of $G$.
 
 Finally $(G/W)/C$ is denoted $D$ and called the {\em semisimple with no definably compact parts} part of $G$. 

Note that  $R/W$ is the connected component of the centre of $C$ and 
$$1 \to R/W \to C \to C_{1} \to 1$$  definably almost splits by results from \cite{HPP}.

\vspace{5mm}
\noindent
One natural question is whether there is a better decomposition theorem.

\begin{Definition} We will say that $G$ has a {\em good decomposition}, if, with above notation, the exact sequence 
$1 \to C \to G/W \to D \to 1$ definably almost splits, namely $G/W$ can be written as $C\cdot D_{2}$ for some definable, definably connected, subgroup $D_{2}$ of $G/W$ which is semisimple with no definably compact parts (i.e. the map $D_{2} \to D$ is surjective with finite kernel).
\end{Definition}

\begin{Lemma} The following are equivalent:
\newline
(i)  $G$ has a good decomposition.
\newline
(ii)  $\pi^{-1}(D_{1})$ is an almost direct product of  $R/W$ (the connected component of its centre) and a definable semisimple group (again necessarily without definably compact parts).
\end{Lemma}
\begin{proof}
This is clear, because  $G/W$ will the almost direct product of $C$ and some $D_{2}$ if and only if  $\pi^{-1}(D_{1})$ is the almost direct product of $R/W$ and $D_{2}$.
\end{proof}

Hence the existence of good decompositions depends on the definable almost splitting of central extensions of semisimple groups without definably compact parts by definable compact groups.

\begin{Remark}  $G$ has a good decomposition in either of the cases:
\newline
(i) $G$ is linear, namely a definable, in $M$, subgroup of some $GL(n,K)$, or
\newline
(ii) $G$ is algebraic, namely of the form  $H(K)^{0}$ for some algebraic group $H$ defined over $K$. 
\end{Remark}
\begin{proof} In fact in both cases (i) and (ii), we point out that $G$ has a {\em definable Levi decomposition}, namely $G$ is an almost semidirect product of its solvable radical $R$ and a definable semisimple group, and this clearly implies that $G$ has a good decomposition. When $G$ is linear this is Theorem 4.5 of \cite{PPSIII}. 
\newline
Suppose now that $H$ is a connected algebraic group defined over $K$, and $G = H(K)^{0}$. We have Chevalley's theorem for $H$ yielding the following exact sequence of connected algebraic groups defined over $K$:
\newline
$$ 1 \to  L \to H \to_{f} A \to 1$$ where $L$ is linear and $A$ is an abelian variety. Then 
$f(G)$ is a connected semialgebraic subgroup of $A(K)$ so is definably compact and commutative, and the semialgebraic connected component of the group of $K$-points of $L$ is a definably connected definable subgroup of $GL(n,K)$ for some $n$. Namely at the level now of definable, definably connected,  groups in $M$, we have an exact sequence
$$ 1 \to R \to G \to_{f} B \to 1$$ where  $R$ is linear, and $B$ is commutative (and definably compact). Again by \cite{PPSII}, $R$ is an almost semidirect product of a definably connected solvable group $R_{1}$ and a definable semisimple group $S$. Let $R$ be the solvable radical of $G$ (as a definable group). As $G/R$ is semisimple, $R$ must map onto $B$ under $f$, whereby $G$ is the almost direct product of $R$ and $S$. 
\end{proof}

\vspace{5mm}
\noindent
Finally in this section we give:

\begin{Example} There is a (Nash) group $G$ without a good decomposition. $T$ will be $RCF$, $M$ the standard model $(\R,+,\cdot)$, and $G$ a certain amalgamated central product of $SO_{2}(\R)$ with the universal cover of $SL_{2}(\R)$.
\end{Example}

\noindent
The model-theoretic setting is the structure $M = (\R,+,\cdot)$. Let $H$ be the definable group $SL_{2}(\R)$ consisting of $2$-by-$2$ matrices over $\R$ of determinant $1$. 
Let $\tilde H  = \widetilde{SL_{2}(\R)}$ be the universal cover of $H$. $\tilde H$ is a connected, simply connected Lie group and we have the exact sequence (of Lie groups) $$1 \to \Z \to \tilde H \to_{\pi} H\to 1$$  where $\Z$ is the discrete group $(\Z,+)$. $\tilde H$ is not definable in $M$, but we will make use of a certain description from section 8.1 of \cite{HPP} (see Theorem 8.5 there)  of $\tilde H$ as a group definable in the $2$-sorted structure $((\Z,+),M)$, and this will be used again in the next section:  
\begin{Fact} There is a $2$-cocycle $h:H\times H \to \Z$ with finite image which is moreover definable in $M$ (in the sense that for each $n\in Im(h)$, $\{(x,y)\in H\times H, h(x,y) = n\}$ is definable in $M$), and such that the set $\Z\times H$ with group structure $(t_{1},x_{1})*(t_{2},x_{2}) =  (t_{1} + t_{2} + h(x_{1},x_{2}), x_{1}x_{2})$ and projection to the second coordinate, is isomorphic to  the group $\tilde H$ with its projection $\pi$ to $H$. 
\end{Fact}

Although not needed, let us say a few words of where the cocycle $h$ comes from, referring to \cite{HPP} for more details. The group $\tilde H$ is naturally ind-definable in $M$, namely as an increasing union $\bigcup_{i}X_{i}$ of definable sets with group operation and projection $\pi$ to $H$ piecewise definable. For some $i$, the restriction of $\pi$ to $X_{i}$ is surjective and as $M$ has Skolem functions there is a definable section $s:H \to X_{i}$ of $\pi|X_{i}$. Define $h$ 
on $H\times H$ by $h(x,y) = s(x)s(y)s(xy)^{-1}$. Then $h$ is as required. 

\vspace{2mm}
\noindent
Let now consider the circle group $SO_{2}(\R)$ and we use additive notation for it. Let $g\in SO_{2}(\R)$ be an element of infinite order. Define a group operation $*$ on
$SO_{2}(\R) \times H$ by  $(t_{1},x_{1})*(t_{2},x_{2}) = (t_{1} + t_{2} + h(x_{1},x_{2})g, x_{1}x_{2})$. Let $G$ be the resulting group, and note that $G$ is now definable (without parameters, taking $g$ algebraic) in $M$. Note that $\{(ng,x): n\in \Z, x\in H\}$ is a subgroup of $(G,*)$ isomorphic to $\tilde H$  (with again projection on second coordinate corresponding to $\pi:\tilde H \to H$). So identifying $\langle g \rangle$ with $\Z$, we have that 
\newline
(i) $SO_{2}(\R)$ is central in $(G,*)$,
\newline
(ii) $G = SO_{2}(\R)\cdot{\tilde H}$,
\newline
(iii) $SO_{2}(\R) \cap {\tilde H} = \Z$,
\newline
and we have the exact sequence of definable, definably connected,  groups in $M$,
$$1 \to SO_{2}(\R) \to G \to H \to 1$$  (where remember $H = SL_{2}(\R)$).
\newline
$H$ is of course definably almost simple and not (definably) compact, whereas $SO_{2}(\R)$ is (definably) compact and central in $G$. To show
 that $G$ does not have a good decomposition it suffices to show that the exact sequence above does not definably almost split in $M$. In fact there is no (even abstract) subgroup $H_{1}$ of $G$ such that $SO_{2}(\R)\cap H_{1}$ is finite and  $SO_{2}(\R)\cdot H_{1} = G$, for otherwise (as $SO_{2}(\R)$ is central in $G$), the commutator subgroup $[G,G]$ is contained in $H_{1}$ so has finite intersection with $SO_{2}(\R)$. But, using (ii) above and the fact that $\widetilde{SL_{2}(\R)}$ is perfect, $[G,G] = \tilde H$ and so has infinite intersection with $SO_{2}(\R)$, a contradiction.
We have completed the exposition of Example 2.10.

\vspace{5mm}
\noindent
In the next section an elaboration of the above analysis will show that passing to a saturated elementary extension, $G^{00} \neq G^{000}$.

%NEW
\begin{Remark} A definably connected group $G$ with a good decomposition does not have necessarily a definable Levi decomposition. 
%$G$ has a good decomposition if and only if $G/W$ has a definable Levi decomposition if and only if the commutator subgroup $[G/W, G/W]$ is definable.
\end{Remark}
\begin{proof}
If one replaces $SO_2(\R)$ with $(\R, +)$ in Example 2.10, then one obtains a group with a good decomposition ($G/W = SL_2(\R)$), but without a definable Levi decomposition (for the same reason as in Example 2.10).
%Let us recall that by \cite{HPP} $C$ can be decomposed as $C = Z(C)^0 [C, C]$, where the commutator subgroup $[C, C]$ is definable and semisimple.

%Clearly, $G/W$ can be written as an almost semidirect product of $C$ and some semisimple with no definably compact parts $D_2$, if and only if $S = [C, C]D_2$ is a definable semisimple subgroup of $G/W$ with $G/W = R/W S = Z(C)^0 [C, C]D_2$ if and only if the commutator subgroup $[G/W, G/W]$ can be decomposed as $[G/W, G/W] = [C, C] D_2$. 
\end{proof}

%%%%%%%%%%%%%%%%%%%%%%%%%%%%%%%%%%%%%%%%%%
\section{$G^{00}$, $G^{000}$ and the examples}
We will first repeat the definitions and geneses of the various notions of ``connected components" of a definable group. 
To begin with let $T$ be an arbitrary complete theory.  We can identify a definable set with the formula $\phi(x)$ which defines it, or rather the functor taking $M$ to $\phi(M)$ from the category $Mod(T)$ (of models of $T$ with elementary embeddings) to $Set$ given by that formula. If the formula has parameters from a set $A$ in a given model of $T$, then the functor is from $Mod(Th(M,a)_{a\in A})$ to $Set$. Likewise for type-definable sets, and also hyperdefinable sets  (a type-definable set quotiented by a type-definable equivalence relation). If $X$ is a type-definable set over $A\subseteq M$, then we sometimes identify $X$ with its interpretation in an $|A|^{+}$-saturated model ${\bar M}$ containing $M$. If $X$ is a type-definable (over $A$) set, defined by partial type $\Phi(x)$ and $E$ a type-definable (over $A$) equivalence relation on $X$ given by partial type $\Psi(x,y)$ then we say that $X/E$ is ``bounded" if  $|\Phi(N)/\Psi(N)|$ is bounded as the model $N$ (containing $A$)  varies. If $X/E$ is bounded it is not hard to see that $|\phi(N)/\Psi(N| \leq 2^{|T|+|A|}$ for all $N$, and if $N_{1}< N_{2}$ are $|A|^{+}$-saturated models containing $A$ then the natural embedding of $\Phi(N_{1})/\Psi(N_{1})$ in $\Phi(N_{2})/\Psi(N_{2})$ is a bijection. In fact, assuming $X/E$ bounded, for a fixed model $M$ containing $A$, and $N$ a saturated model containing $M$, the $E$-class of some $b\in X$ depends only on $tp(b/M)$, hence the map $X\to X/E$ factors through the space  $S_{\phi}(M)$ of complete types over $M$ extending $\phi(x)$.
Equipped with the quotient topology (which we call the logic topology), $X/E$ is a compact Hausdorff space. 

Now suppose that the equivalence relation $E$ on $X$ is given instead by a possibly infinite disjunction $\bigvee_{i} \Psi_{i}(x,y)$ of partial types over $A$ (i.e. working in a saturated model ${\bar M}$, is $Aut({\bar M}/A)$-invariant, or as we often just say $A$-invariant). The whole discussion above regarding boundedness of $E$ goes through in this more general case, including the fact that the map $X \to X/E$ factors through the type space $S_{\Phi}(M)$ (for $M$ any model containing $A$) However the quotient topology on $X/E$ is no longer Hausdorff, and it is probably better to view $X/E$ as an object of descriptive set theory or maybe even noncommutative geometry.

Let us first consider the case where $X$ is a sort of $T$. Then given any (small) set $A$ of parameters, there is a finest bounded type-definable over $A$ equivalence relation on $X$ which we call $E_{X,A,KP}$. Likewise there is finest bounded $A$-invariant equivalence relation on $X$ which we call $E_{X,A,L}$. For $a\in X$, the $KP$-strong type of $a$ over $A$ is precisely the $E_{X,A,KP}$-class of $a$, and the Lascar strong type of $a$ over $A$ is precisely the $E_{X,A,L}$-class of $a$. There is also of course the usual strong type of $a$ over $A$, which is the $E_{X,A,Sh}$-class of $a$ where $E_{X,A,Sh}$ is the intersection of all $A$-definable equivalence relations on $X$ with finitely many classes. In stable theories all these strong types coincide. In \cite{CLPZ} an example was given where $KP$-strong types differ from Lascar strong types. More (natural) examples will be given later.

We now consider the case where $X = G$ is a definable group, and $E$ comes from an appropriate subgroup of $G$.  So we assume $G$ to be a group definable in a saturated model ${\bar M}$, and we fix a small set $A$ of parameters over which $G$ is defined. $G_{A}^{0}$ denotes the intersection of all $A$-definable subgroups of $G$ of finite index. It is clearly a type-definable (normal) subgroup of $G$ of bounded index, and equipped with the logic topology the quotient $G/G_{A}^{0}$ is a profinite group. We let $G_{A}^{00}$ denote the smallest type-definable over $A$ subgroup of $G$ of bounded index. It is also normal, the quotient $G/G_{A}^{00}$, equipped with the logic topology is a compact (Hausdorff) topological group, and $G/G_{A}^{0}$ is its maximal profinite quotient. Finally $G_{A}^{000}$ is the smallest $A$-invariant subgroup of $G$, of bounded index, which is again normal. 
We have that
$G_{A}^{000} \leq  G_{A}^{00} \leq  G_{A}^{0}$. 

A well-known construction (see \cite{Gismatullin-Newelski}) links these different ``connected components" of definable groups with the various strong types. We give a simplified version: Let $T$ be a complete theory such that $dcl(\emptyset)$ is a model. Let $G$ be a $\emptyset$-definable group. 
Adjoin a new sort $S$ together with a regular action of $G$ on $S$. Call the new theory $T'$. Clearly no ``new structure" is imposed on $T$. Work in a saturated model of $T'$. Then 
\begin{Fact} (i) $E_{S,\emptyset,Sh}$ is the orbit equivalence relation on $S$ induced by $G_{\emptyset}^{0}$.
\newline 
(ii) $E_{S,\emptyset,KP}$ is the orbit equivalence relation on $S$ induced by $G_{\emptyset}^{00}$, and 
\newline
(iii) $E_{S,\emptyset,L}$ is the orbit equivalence relation on $S$ induced by $G_{\emptyset}^{000}$.
\end{Fact}

Hence, if for example $G^{00} \neq G^{000}$, then we obtain in this way examples where $KP$-strong type differs from Lascar strong type.

%\begin{Remark} Suppose $f:G\to H$ is a surjective $\emptyset$-definable homomorphism between $\emptyset$-definable groups. Then 
%$f(G_{\emptyset}^{0}) = H_{\emptyset}^{00}$, $f(G_{\emptyset}^{00}) = H_{\emptyset}^{00}$ and $f(G_{\emptyset}^{000}) = H_{\emptyset}^{000}$.
%\end{Remark}
%\begin{proof}  Clear. For example if  $L$ is a type-definable over $\emptyset$ bounded index subgroup of $H$, properly contained in $f(G_{\emptyset}^{00})$ then $f^{-1}(L)$ is a type-definable over $\emptyset$ bounded index subgroup of $G$, properly contained in $G_{\emptyset}^{00}$, a contradiction.
%\end{proof}

There are plenty of examples where $G_{\emptyset}^{0} \neq G_{\emptyset}^{00}$ (such as definably compact groups definable in $o$-minimal structures). However, until now no examples had been worked out where $G_{\emptyset}^{00} \neq G_{\emptyset}^{000}$.

\noindent
We say, for example, that ``$G^{0}$ exists" if for some set $A$ of parameters, for all $B\supseteq A$, $G_{A}^{0}= G_{B}^{0}$. If $G^{0}$ exists, then,
assuming $G$ is $\emptyset$-definable, we can take $A$ to be $\emptyset$ and we define $G^{0}$ to be 
$G_{\emptyset}^{0}$.  Likewise for $G^{00}$ and $G^{000}$.  If $G^{000}$ exists then so do  $G^{00}$ and $G^{0}$. 
Gismatullin \cite{Gismatullin} proves, following work of Shelah, that if $T$ has $NIP$ then for any definable group $G$,
$G^{000}$ exists. When $T$ is stable,  $G^{0} = G^{00} = G^{000}$. For $T$ simple, $G^{0}$ may not exist, but it is known that for any $A$, $G_{A}^{00} = G_{A}^{000}$. It is conjectured (for $T$ simple) that $G_{A}^{0} = G_{A}^{00}$ and this is known in the supersimple case (\cite{Wagner}). 

When we are working with either $o$-minimal theories, or closely related $NIP$ theories, we just say $G^{0},G^{00}$, $G^{000}$. 

\vspace{2mm}
\noindent
We now give examples of $G$ (including $o$-minimal examples) where $G^{00} \neq G^{000}$.  In the sequel to this paper we will make a systematic analysis of $G^{00}$ and $G^{000}$ in the $o$-minimal case, showing that the behaviour in Theorem 3.3 for example is typical. 

\begin{Theorem} Let $T = Th(\widetilde{SL_{2}(\R)},\cdot)$. Then $T$ has $NIP$, and if $(G,\cdot)$ denotes a saturated model, then 
$G^{00}\neq G^{000}$. In fact $G = G^{00}$ and $G/G^{000}$ is isomorphic to  ${\widehat \Z}/\Z$ where $\widehat \Z$ is the profinite completion of  $(\Z,+)$.
\end{Theorem}

\begin{proof} From Fact 2.11  and the discussion following it (taken from \cite{HPP}) the group $(\widetilde{SL_{2}(\R)},\cdot)$ is interpretable (with parameters) in the $2$-sorted structure  
\newline
$$((\Z,+), (\R,+,\times))$$  
(where there are no additional basic relations between the sorts). As $Th(\Z,+)$ is stable (in fact superstable of $U$-rank $1$) and $RCF$ has $NIP$ clearly the $2$-sorted structure has $NIP$ too, and hence the interpretable group
$(\widetilde{SL_{2}(\R)},\cdot)$ has $NIP$.

\noindent
In fact we will work with $T = Th((\Z,+),(\R,+,\times))$ (rather than $Th(\widetilde{SL_{2},\R)},\cdot)$, and will point out how the results are also valid for the ``pure group structure". 

\noindent
Let $M$ denote  $(\R,+,\cdot)$, and $N$ denote the $2$-sorted structure $((\Z,+),(\R,+,\times))$. Then a saturated model ${\bar N}$ of $T$ will be of the form $((\Gamma,+), \bar M)$ where ${\bar M}$ is a saturated real closed field $(K,+,\cdot)$ say, and $(\Gamma,+)$ is a saturated elementary extension of $(\Z,+)$.  
Let now $G$ denote the interpretation in the big model ${\bar N}$ of the formula(s) defining the group $\widetilde{SL_{2}(\R)}$ in $N$. So clearly, using Fact 2.11,  
$G$ has universe  the definable set $\Gamma \times SL(2,K)$  and group operation given by  $(t_{1},x_{1})*(t_{2},x_{2}) = (t_{1} + t_{2} + h(x_{1},x_{2}),x_{1}x_{2})$.
Here $h(x_{1},x_{2})\in \Z < \Gamma$ so everything makes sense. We write the group $G$ as $(G,\cdot)$ hopefully without ambiguity.
We identify the group $\Gamma$ with the subgroup  $(\{(t,1):t\in\Gamma\},*)$ of $G$ via the (definable) isomorphism $\iota$ which takes $t\in \Gamma$ to $(t-h(1,1),1)\in G$.   As such $\Gamma$ is central in $G$ and we have the exact sequence

  \begin{equation}1 \to\Gamma \to G \to SL(2,K)\to 1 \end{equation}
 
We again identify $\Z < \Gamma$ with the subgroup  $(\{(t,1):t\in \Z\},*)$ of $G$ via $\iota$. 
Note that  $(\{(t,x):t\in \Z, x\in SL(2,K)\},*)$ is a (non definable) subgroup of $G$, which we will take the liberty to call  
$\widetilde{SL_{2}(K)}$. (In fact it will identify with the so-called $o$-minimal universal cover of  $SL(2,K)$, an ind-definable group in ${\bar M}$, but this fact will not be needed.). From (1) we obtain:

\begin{equation} 1 \to \Z \to \widetilde{SL_{2}(K)} \to SL_{2}(K) \to 1  \end{equation}
(where only $SL_{2}(K)$ is definable). 

So with the above identifications we write 

\begin{equation}  G = \Gamma \cdot \widetilde{SL_{2}(K)} \end{equation}
   where the subgroup $\Gamma$ of $G$ is definable and central, the subgroup $\widetilde{SL_{2}(K)}$ of $G$ is not definable and  $\Z = \Gamma \cap \widetilde{SL_{2}(K)}$.

\vspace{2mm}
\noindent
We now aim to understand $G^{000}$ in terms of this decomposition (even though $\widetilde{SL_{2}(K)}$ is not definable).  

\vspace{2mm}
\noindent
{\em Claim 1.} $\Gamma^{000} = \Gamma^{00} = \Gamma^{0} = \bigcap_{n}n\Gamma$, and is contained in $G^{000}$.
\newline
{\em Proof of Claim 1.} $\Gamma$ (as a group definable in $N$) is simply a model of $Th(\Z,+)$ which is stable, so we have equality of the various connected components and $\Gamma^{0}$ is the intersection of all definable subgroups of finite index which is as described. Also $G^{000}\cap \Gamma$ clearly contains $\Gamma^{000}$.
\newline
{\em End of proof.}

\vspace{2mm}
\noindent
{\em Claim 2.}  $\widetilde{SL_{2}(K)}$ is perfect, namely equals its own commutator subgroup.
\newline
{\em Proof of Claim 2.}  Because of the exact sequence (2) above and the well-known fact that $SL_{2}(K)$ is perfect, it is enough to show that the subgroup $\Z$ of $\widetilde{SL_{2}(K)}$  is contained in 
$[\widetilde{SL_{2}(K)},\widetilde{SL_{2}(K)}]$. But this follows immediately because $\Z$ is contained in the (naturally embedded) subgroup  $\widetilde{SL_{2}(\R)}$ of $\widetilde{SL_{2}(K)}$, and again $\widetilde{SL_{2}(\R)}$ is known to be perfect. 
\newline
{\em End of proof.}

\vspace{2mm}
\noindent
{\em Claim 3.}  $\widetilde{SL_{2}(K)} \subseteq G^{000}$.
\newline
{\em Proof of Claim 3.} Let $H = \widetilde{SL_{2}(K)} \cap G^{000}$. $H$ is then a normal subgroup of $\widetilde{SL_{2}(K)}$ of index at most the continuum.
Hence $\pi(H)$ the image of $H$ under $\pi:\widetilde{SL_{2}(K)} \to SL_{2}(K)$ is an infinite normal subgroup of $SL_{2}(K)$. As $SL_{2}(K)$ is
simple as an abstract group modulo its finite centre, and is also perfect, it follows that $\pi(H) = SL_{2}(K)$. Hence
$\widetilde{SL_{2}(K)} = \Z\cdot H$, and as $\Z$ is central, the commutator subgroup of $\widetilde{SL_{2}(K)}$ is contained in $H$. By Claim 2, $H = \widetilde{SL_{2}(K)}$, as required.
\newline
{\em End of proof.}

\vspace{2mm}
\noindent
(Note that we have shown that $\widetilde{SL_{2}(K)}$ has {\em no} proper normal subgroup not contained in its centre.)

\vspace{2mm}
\noindent
{\em Claim 4.}  $[G,G] = \widetilde{SL_{2}(K)}$
\newline
{\em Proof of Claim 4.} 
By the description of $G$ in (3), $[G,G]$ is a subgroup of $\widetilde{SL_{2}(K)}$. By Claim 2, we get equality.
\newline
{\em End of proof.}

\vspace{2mm}
\noindent
{\em Claim 5.}  $G^{000} = \Gamma^{0}\cdot \widetilde{SL_{2}(K)}$
\newline
{\em Proof of Claim 5.}  By Claims 1 and 3, $G^{000}$ contains $\Gamma^{0}\cdot \widetilde{SL_{2}(K)}$.
On the other hand  $\Gamma^{0}\cdot \widetilde{SL_{2}(K)}$ is clearly of bounded index in $G$, and using Claim 4 is also clearly invariant under automorphisms of $N$ which fix the parameters defining $G$.
So we get equality.  In fact note at this point that $\Gamma^{0}\cdot \widehat{SL_{2}(K)}$ is also invariant under automorphisms of  the structure $(G,\cdot)$, so  coincides with $G^{000}$ in this reduct. 
\newline
{\em End of proof.}

\vspace{2mm}
\noindent
{\em Claim 6.}  $G = G^{00}$.
\newline
{\em Proof of claim 6.}  By Claim 5 and (3), $G^{000} \cap \Gamma = \Gamma^{0}\cdot \Z$. So as $G^{000}\subseteq G^{00}$, $G^{00}\cap\Gamma$ contains  $\Gamma^{0}\cdot\Z$ {\em and must} type-definable. This can be directly seen to be a contradiction unless $G^{00}\cap \Gamma = \Gamma$. For example $\Gamma/\Gamma^{0} = \widehat\Z$ the profinite completion of 
$\Z$ and the subgroup $\Z$ of $\Gamma$ goes isomorphically to the dense subgroup $\Z$ of $\hat\Z$ under the quotient map. But then under this quotient map $G^{00}\cap \Gamma$ must go to a closed subgroup of $\widehat\Z$ which contains the dense subgroup $\Z$, hence must go to $\widehat Z$ and so $G^{00}\cap \Gamma = \Gamma$. 
\newline
{\em End of proof.}

\vspace{2mm}
\noindent
We have already seen in the proof of Claim 6 that $G/G^{000}$ is naturally isomorphic to $\widehat\Z/\Z$. So together with Claims 5 and 6 this completes the proof of Theorem 3.2.
\end{proof}

\vspace{5mm}
\noindent 
We now give a similar $o$-minimal example. We will refer at some point in the proof to the fact that for a definably compact group $H$ (such as $SO_{2}$) in a saturated $o$-minimal structure, $H^{00} = H^{000}$ which follows from results in \cite{NIPII}. 
\begin{Theorem}  Let $T$ be $RCF$, and $G$ the group from Example 2.10.  Let $G_{1}$ be $G({\bar M})$ for ${\bar M} = (K,+,\cdot)$ a saturated model. Then $G_{1} = G_{1}^{00}$, but $G_{1} \neq G_{1}^{000}$ and in fact 
$G_{1}/G_{1}^{000}$ is naturally isomorphic to the quotient of the circle group  $SO_{2}(\R)$ by a dense cyclic subgroup.
\end{Theorem} 
\begin{proof}  The proof is more or less identical to that of Theorem 3.2, so we just give a sketch. In analogy with
(3) from the proof of 3.2  and with the same notation we have:
\newline
(*)    $G_{1}$ is a central product of its subgroups $SO_{2}(K)$ (which is definable) and $\widetilde{SL_{2}(K)}$ which is not definable, and with intersection ``$\Z$ " (an infinite cyclic subgroup $\langle g \rangle$ of $SO_{2}(\R)< SO_{2}(K)$). 

\vspace{2mm}
\noindent
As in Claims 3 and 4 in the proof of 3.2, $G_{1}^{000}$ contains $\widetilde{SL_{2}(K)}$, and (using (*)) $[G_{1},G_{1}]  =  \widetilde{SL_{2}(K)}$. Also $G^{000} \cap SO_{2}(K)$ contains $SO_{2}(K)^{000}$ which we know to be equal to $SO_{2}(K)^{00}$. Hence we conclude that
\newline
(**) $G_{1}^{000} = SO_{2}(K)^{00}\cdot\widetilde{SL_{2}(K)}$. 

\vspace{2mm}
\noindent
Now the quotient map $SO_{2}(K) \to SO_{2}(K)/SO_{2}(K)^{00}$ identifies with the standard part map $SO_{2}(K) \to SO_{2}(\R)$ which is the identity on $SO_{2}(\R)$ and in particular on $\langle g \rangle$ (so $\langle g \rangle \cap SO_{2}(K)^{00}$ is trivial).  
\newline
By (**) $G_{1}^{00} \cap SO_{2}(K)$ is type-definable and contains $SO_{2}(K)^{00}\cdot\langle g \rangle$, so its image under the standard part map $SO_{2}(K) \to SO_{2}(\R)$ is a closed subgroup which contains the dense subgroup $\langle g\rangle$, hence has to be $SO_{2}(\R)$. So $G_{1}^{00}$ contains $SO_{2}(K)$ hence by (*) $G_{1}^{00} = G_{1}$. 
\end{proof} 

\vspace{5mm}
\noindent
As remarked earlier the above theorems provide new examples of non $G$-compact theories, i.e. where Lascar strong types differ from $KP$-strong types.
A possibly  interesting question, especially bearing in mind the above examples, is how one can or should view, naturally, $G/G^{000}$ (or even $G^{00}/G^{000}$) as a mathematical object.  For example it is an abstract group, a quasi-compact (compact but not necessarily Hausdorff) topological group, as well as a quotient of a type-space by an $F_{\sigma}$ equivalence relation. In the above examples it is, in a natural fashion, the quotient of a compact (Hausdorff) commutative group by a countable dense subgroup. We will show in the sequel that this is always the case when $G$ is definable in an $o$-minimal structure. A natural problem at this point is to find $G$ such that $G_{\emptyset}^{00}/G_{\emptyset}^{000}$ is noncommutative. Also we see, via the examples above, some relationships between universal covers and fundamental groups on the one hand, and Lascar groups on the other, and maybe the connection is more than just accidental.

\section{Definable amenability and bounded orbits}
We begin with an arbitrary theory $T$. We recall that if $M$ is a model, and $X$ a definable set in $M$, then a Keisler measure $\mu$ on $X$ (over $M$) is a finitely additive probability measure on the family of subsets of $X$ which are definable (with parameters) in $M$. A Keisler measure $\mu$ on $X$ over $M$ induces and is induced by a (unique) regular Borel probability measure on the space $S_{X}(M)$ of complete types over $M$ containing the formula defining $X$, which we sometimes identify with $\mu$. (See the introduction to Section 4 of \cite{NIPII}.) In fact a Keisler measure on $X$ over $M$ should be seen as a generalization of a complete type over $M$ (which contains the formula $``x\in X"$). 

When $X = G$ is a definable group, namely is equipped with a definable group structure, then $G(M)$ acts (on both the left and right) on the set (in fact space) of Keisler measures $\mu$ on $G$ over $M$: if $Y$ is an $M$-definable subset of $G$ then, 
$(g\cdot\mu)(Y) = \mu(g^{-1}\cdot Y)$. In particular it makes sense for a Keisler measure $\mu$ on $G$ over $M$ to be left (or right) 
$G(M)$-invariant. If $G$ has such a left $G$-invariant Keisler measure then we say that $G$ is {\em definably amenable}. 

In fact (assuming $G$ is definable without parameters), this is a property of $Th(M)$, in the sense that if $N$ is another model of $T$ and $G(N)$ is the interpretation in $N$ of the formulas defining $G$, then $G(M)$ is definably amenable iff $G(N)$ is. This follows from Proposition 5.4 of \cite{NIPI}. 

In the above context we also have the (left and right) actions of $G(M)$ on the space $S_{G}(M)$ (of completes types over $M$ concentrating on $G$). 
When $M$ is a ``big" model, and $p(x)\in S_{G}(M)$, we have the notion ``$p$ has bounded orbit" from \cite{Newelski} for example. We will take our working definition as the following rather crude one, which on the face of it depends on set theory.
\begin{Definition} Suppose $\bar\kappa$ is an inaccessible cardinal, and ${\bar M}$ a saturated model of cardinality $\bar\kappa$. 
\newline
(i) We will say that {\em $p(x)\in S_{G}({\bar M})$ has bounded orbit} if the orbit of $p$ under the (left) action of $G({\bar M})$ is of cardinality 
$<\bar\kappa$, equivalently if $Stab(p) = \{g\in G({\bar M}): gp = p\}$ is a subgroup of $G({\bar M})$ of index $< \bar\kappa$. 
\newline
(ii) We say that {\em $G$ has a bounded orbit} if some $p(x)\in S_{G}({\bar M})$ has bounded orbit.
\end{Definition}

In \cite{Newelski} some more careful definitions (see Definition 1.1 there) are given of ``bounded orbit" avoiding the dependence on set theory (and some problems are mentioned concerning the possible sizes of bounded orbits), and our results in this section hold with these more refined definitions. The same paper \cite{Newelski} states a conjecture attributed to Petrykowski:
\begin{Conjecture} If $G$ has a bounded orbit then $G$ is definably amenable.
\end{Conjecture} 

As discussed in the introduction the motivation for this conjecture seems to be also closely connected to $G^{00}$ and $G^{000}$, in the sense that one may hope, given a global type $p$ with bounded orbit, to be able to show that $G^{00} = G^{000} = Stab(p)$ and then to $p$ to lift the Haar measure on $G/G^{00}$ to a translation invariant Keisler measure on $G$. 
The aim of this section is to prove Conjecture 4.2 in the $o$-minimal context (although we have not yet ``identified" those types with bounded orbit).  
We do this by characterizing each of the properties ``definable amenability" and ``having a bounded orbit" in terms of the decomposition given in Proposition 2.6 and concluding that they coincide. So in a sense it is a proof by inspection.

We first describe when a definable group in an $o$-minimal structure is definably amenable.  The proof is basically due to Hrushovski.

We begin with some preparatory lemmas, the first two of which are in a general context. 

\begin{Lemma} Suppose $T$ has definable Skolem functions. Let $G$ be definable and definably amenable. Then any definable subgroup $H$ of $G$ is also definably amenable.
\end{Lemma}
\begin{proof} Let $\mu$ be a left $G$-invariant Keisler measure on $G$.  By the existence of definable Skolem functions there is a definable subset $S$ of $G$ which meets each coset of $H$ in $G$ in exactly one point. Define $\lambda$ on definable subsets of $H$ by:  for $Y$ a definable subset of $H$,
$\lambda(Y) = \mu(Y\cdot S)$  where  $Y\cdot S = \{a\cdot b: a\in Y, b\in S\}$. 
\newline
It is easy to see that $\lambda$ is a Keisler measure on $H$. Left $H$-invariance, is because, for $Y\subseteq H$ definable and $h\in H$, $\lambda(h\cdot Y) = \mu((h\cdot Y)\cdot S) = \mu(h\cdot(Y\cdot S)) =$ (by left invariance of $\mu$) $\mu(Y\cdot S)  = \lambda(Y)$.

\end{proof}

\begin{Lemma} Suppose $G$ is definable and $H$ is a definable normal subgroup.
\newline
(i) If $G$ is definably amenable, so is $G/H$.
\newline
(ii) (Assume $T$ has $NIP$.) If both $H$ and $G/H$ are definably amenable, so is $G$.
\end{Lemma}
\begin{proof}
(i) Let $\pi:G\to G/H$ be the canonical surjective homomorphism. If $\mu$ is a left $G$-invariant Keisler measure on $G$, then the ``pushforward measure" on $G/H$ defined by $\lambda(Y) = \mu(\pi^{-1}(Y))$ is a left invariant Keisler measure on $G/H$.
\newline
(ii) We work in a saturated model ${\bar M}$. Let $\mu,\lambda$ be translation-invariant Keisler measures on $H$ and $G/H$ respectively over ${\bar M}$ (i.e. ``global" Keisler measures). By Lemma 5.8 of \cite{NIPII} we may assume that $\mu$ is {\em definable}. We define a global Keisler measure $\chi$ on $G$ by integration:
Namely, let $X$ a definable subset of $G$, and we may assume that both $X$ and $\mu$ are definable over a small model $A$. For $g/H \in G/H$, let $f(g/H) = \mu((g^{-1}X) \cap H)$, noting by translation invariance of $\mu$, that this is well-defined. By definability of $\mu$ over $M$, $f(g/H)$ depends on $tp((g/H)/M)$ and the corresponding map from the relevant space of complete types $S_{G/H}(M)$  to  $[0,1]$ is continuous. So considering $\lambda$ as inducing a Borel measure on $S_{G/H}(M)$ we can form $\int f d\lambda$, which we define to be $\chi(X)$. It is easily checked that $\chi$ is a global translation invariant Keisler measure on $G$. 
\end{proof}

\begin{Lemma} Suppose $G$ is a definably almost simple, non definably compact group, definable in an $o$-minimal expansion $M$ of a real closed field $K$ say. Then $G$ is not definably amenable.
\end{Lemma}
\begin{proof} The main point is to observe that, working up to definable isogeny, $G$ contains a definable subgroup definably isomorphic to $PSL(2,K)$. 
Granting this observation, the lemma follows from Lemma 4.4 together with Remark 5.2(iv) of \cite{NIPI} (which states that $PSL(2,K)$ is not definably amenable). The observation itself follows from results in \cite{PPSI} and \cite{PPSIII}, together with the classification of the real simple Lie algebras corresponding to simple noncompact Lie groups.
\end{proof}

We can now conclude, where notation comes from the paragraph following the proof of Proposition 2.6.
\begin{Proposition} Let $G$ be a definable, definably connected, group in an $o$-minimal expansion $M$ of a real closed field. Then $G$ is definably amenable if and only if $D$ (the semisimple with no definably compact parts, part of $G$)  is trivial. 
\end{Proposition}
\begin{proof} First suppose that  $D$ is trivial, so we have a short exact sequence $$1\to W \to G \to C \to 1$$ where $W$ is solvable and $C$ is definably compact. Now $W$ is amenable as an abstract group, so in particular definably amenable, and by \cite{NIPII}, $C$ is definable amenable. As $Th(M)$ has $NIP$, by Lemma 4.4(ii) $G$ is definably amenable.

Conversely, if  $G$ is definably amenable, then by Lemma 4.4(i), $D$ is too, as it is a quotient of $G$. If $D$ is nontrivial then it contains a definably almost simple (non definably compact) definable subgroup, which by Lemma 4.3 is definably amenable. This contradicts  Lemma 4.5.
\end{proof}

We give a little more information around definable amenability by noting:
\begin{Proposition} ($T$ an $o$-minimal expansion of $RCF$.) Suppose $G$ is definable, definably connected, and torsion-free. Then $G$ has a (left) invariant, definable, global complete type. 
\end{Proposition}
\begin{proof} We again argue by induction on $dim(G)$. By Proposition 1.1 (i), $G$ contains a normal definable subgroup $H$ such that $G/H$ is $1$-dimensional. From results in \cite{Razenj} we may assume that $G/H$ is an open interval in $1$-space with continuous group operation. The global type at ``$+\infty$", $p$ say, is both definable and translation invariant. On the other hand the induction hypothesis gives a definable  translation invariant global complete type $q$ of $H$. The argument (by integration) in the proof of Lemma 4.4(ii) produces a global complete type of $G$ which is both translation invariant and definable.
\end{proof}

\vspace{2mm}
\noindent
We now focus on Conjecture 4.2. 
From now on ${\bar M}$ denotes a saturated model of (arbitrary complete countable) $T$, of cardinality $\bar\kappa$ where $\bar\kappa$ is inaccessible, and $G$ an $\emptyset$-definable group. 
Let us first remark that the converse to Conjecture 4.2 holds for $NIP$ theories.

\begin{Remark} (Assume $T$ has $NIP$.) Suppose $G$ is definably amenable. Then $G$ has a bounded orbit.
\end{Remark}
\begin{proof} By Proposition 5.12 of \cite{NIPII}, $G$ has a global $f$-generic type $p$. Fix a small model $M_{0}$ which witnesses this. There will then be a bounded number of global complete types which do not fork over $M_{0}$, as there are a bounded number of complete types over $M_{0}$, and by $NIP$ any complete type over $M_{0}$ has a bounded number of global nonforking extensions (Proposition 2.1 of \cite{NIPII}). As every $G({\bar M})$-translate of $p$ does not fork over $M_{0}$ there are a bounded number of such translates so $p$ has bounded orbit.
\end{proof}

\begin{Lemma} Suppose $G = G({\bar M})$ is almost simple as an abstract group, in the sense that $G$ has no infinite proper normal subgroups.  Then $G$
has no proper subgroup of index $<\bar\kappa$. In particular any bounded orbit of $G$ is a singleton (namely a translation invariant type).
\end{Lemma}
\begin{proof} Suppose $H$ were a proper subgroup of $G$ of bounded index. Then $G$ acts transitively on the homogeneous space $X = G/H$. Let 
$N = \{g\in G:gx = x$ for all $x\in X\}$. Then $N$ is a proper normal subgroup of $G$. As $G/N$ acts faithfully on $X$ and $|X|< \bar\kappa$, also 
$|G/N| < \bar\kappa$, in particular $N$ is an infinite proper normal subgroup of $G$, contradiction.
\newline
For the ``in particular" clause: if $p\in S_{G}({\bar M})$ has bounded orbit, then $Stab(p)$ is a subgroup of $G$ of bounded index. By what has just been shown $Stab(p) = G$ so $p$ is left $G$-invariant.
\end{proof}

\begin{Lemma} Let $f:G\to H$ be a definable surjective homomorphism. If  $G$ has a bounded orbit, so does $H$.
\end{Lemma}
\begin{proof} Let $p\in S_{G}({\bar M})$ have bounded orbit. Then $q = f(p)\in S_{H}({\bar M})$, and if $g\in Stab_{G}(p)$ then 
$q = f(p) = f(gp) = f(g)q$ hence $f(Stab_{G}(p)) \subseteq Stab_{H}(q)$. As $Stab_{G}(p)$ has bounded index in $G$, also $Stab_{H}(q)$ has bounded index in $H$. 
\end{proof}

\begin{Proposition} Assume $T$ is an $o$-minimal expansion of $RCF$ and $G$ is definably connected. Suppose $G$ has a bounded orbit. Then $D$ (the semisimple with no definably compact parts, part of $G$) from Proposition 2.6 is trivial.
\end{Proposition}
\begin{proof} Suppose for a contradiction that $D$ is nontrivial. Then $D$ is an almost direct product of definable, definably almost simple non definably compact groups $D_{i}$. But then for $i=0$ say there is  a definable surjective homomorphism $f$ from $G$ to $D_{0}$. By Lemma 4.10, $D_{0}$ has a bounded orbit. As remarked earlier (Corollary 6.3 of \cite{PPSIII}) $D_{0}$ is almost simple as an abstract group, so by Lemma 4.9, $D_{0}$ has an invariant (global) type. This contradicts non definable amenability of $D_{0}$ (Lemma 4.5).
\end{proof}

\begin{Corollary} ($T$ an $o$-minimal expansion of $RCF$). $G$ has a bounded orbit if and only if $G$ is definably amenable.
\end{Corollary}
\begin{proof} If $G$ has a bounded orbit, then by Proposition 4.11 and Proposition 4.6, $G$ is definably amenable. The converse is Remark 4.8.
\end{proof}

\vspace{5mm}
\noindent
Finally we discuss a strengthening of Conjecture 4.2 in which we try to describe bounded orbits themselves. As we are not completely sure which way it will go we state the new conjecture as a question (with notation as above).
\begin{Problem} (Assume $T$ has $NIP$.)  Is it the case that $p\in S_{G}({\bar M})$ has bounded orbit (equivalently stabilizer of bounded index) if and only if  $p$ is $f$-generic?
\end{Problem}

Again the right to left direction holds with proof contained in the proof of Remark 4.8. 
In the $o$-minimal case we hope to give an explicit description of global types with bounded orbit from which a positive answer to Problem 4.13 can be just read off. By Corollary 4.12 and Proposition 4.6  we may restrict ourselves to definable groups $G$ for which $D$ (from the discussion after Proposition 2.6) is trivial, hence $G$ is built up from a definably compact group, and $1$-dimensional torsion-free groups. Here we just point out that Problem 4.13 has a positive answer for these constituents, and leave the general ($o$-minimal case) to later work. For the next lemma we recall that a
definable subset of $G$ (or the formula defining it) is said to be  left generic if finitely many left translates of $X$ cover $G$. Likewise for right generic. Definably compact groups $G$ in $o$-minimal expansions of real closed fields have the so-called ``finitely satisfiable generics" property (see \cite{NIPI}) which says that there is a global type of $G$ every left translate of which is finitely satisfied in some given small model. The $fsg$ property implies among other things that left genericity coincides with right genericity for definable subsets of $G$, so we just say {\em generic}. A generic type $p\in S_{G}(M)$ is one all of whose formulas are generic, and again such global types exist when $G$ is definably compact in $o$-minimal $T$.

\begin{Lemma} ($T$ $o$-minimal.) Suppose $G$ is definably compact, and $p(x)\in S_{G}({\bar M})$. Then the following are equivalent:
\newline
(i) $p$ has bounded $G$-orbit,
\newline
(ii) $p$ is generic,
\newline
(iii)  $p$ is $f$-generic.
\end{Lemma}
\begin{proof}
In fact the implications $(ii) \to  (iii) \to (i)$ hold for  $fsg$ groups in arbitrary  $NIP$ theories and the proof will be at this level of generality. 
\newline
(iii) implies (i) is given by the proof of Remark 4.8.
\newline
(ii) implies (iii): By \cite{NIPI} (see also Fact 5.2 of \cite{NIPII}), any generic formula $\phi(x)$ over ${\bar M}$ is satisfied in any small model $M_{0}$ (over which $G$ is defined). So if $p\in S_{G}({\bar M})$ is generic, then every left translate of $p$ is finitely satisfied in $M_{0}$  (where $M_{0}$ is any small model over which $G$ is defined), so in particular every left translate of $p$ does not fork over $M_{0}$, hence $p$ is left generic.

\vspace{2mm}
\noindent
(i) implies (ii):  Here we give the proof assuming $o$-minimality of $T$ and definable compactness of $G$. Suppose $p$ is not generic. Let $X$ be a definable set (or formula) in $p$ which is not generic. Note that we may assume $G$ to be a closed bounded definable subset of some ${\bar M}^{n}$. The closure of $X$ in $G$ equals $X\cup Y$ where $dim(Y) < dim(G)$. So $Y$ is not generic in $G$. Hence as the set of non generic definable sets is an ideal, the closure of $X$ is also non generic (and of course in $p$). The upshot is that we may assume $X$ to be closed. Let $M_{0}$ be a small model over which $G$ and $X$ are defined. If for every $g\in G$, the left translate $g\cdot X$ meets $G(M_{0})$, then by compactness $X$ is right generic, so generic, a contradiction. Hence for some $g\in G$, $(g\cdot X) \cap G(M_{0}) = \emptyset$. Now $g\cdot X$ is also closed in $G$. So by results in \cite{Dolich} and \cite{Peterzil-Pillay} (see also \cite{Starchenko}), 
$g\cdot X$ forks over $M_{0}$. By the main result of \cite{Chernikov-Kaplan}  (which is maybe implicit in other papers 
in the $o$-minimal case), $g\cdot X$ divides over $M_{0}$. As $X$ is defined over $M_{0}$ this means that for some $M_{0}$-indiscernible sequence $(g_{i}:i<\omega)$ and 
some $k< \omega$, $\{g_{i}\cdot X: i< \omega\}$  is $k$-inconsistent, in the sense that for every (some)  $i_{1} < .. 
< i_{k}$, $(g_{i_{1}}\cdot X) \cap .... \cap (g_{i_{k}}\cdot X) = \emptyset$.  We can stretch the $M_{0}$
-indiscernible sequence $(g_{i}:i<\omega)$  to $(g_{i}:i< {\bar\kappa})$. So $\{(g_{i}\cdot X):i< {\bar\kappa}\}$ is 
also $k$-inconsistent. It follows easily that among the set $\{g_{i}p:i<{\bar\kappa}\}$ of
of complete global types there are $\bar\kappa$ many distinct types. So $p$ does not have bounded orbit. 
\end{proof}

\vspace{2mm}
\noindent
Let us note that various ingredients of the proof of (i) implies (ii) above also appear in earlier papers such as \cite{NIPII}. In fact there {\em is} a proof of (i) implies (ii) (so of the whole lemma) in the more general context of $fsg$ groups in $NIP$ theories, but depending on some additional machinery. It will appear in a subsequent paper.

\begin{Lemma} Suppose $G$ is $1$-dimensional and torsion-free (divisible), and $p\in S_{G}({\bar M})$. Then the following are equivalent:
\newline
(i) $p$ has bounded $G$-orbit,
\newline
(ii) $p$ is $G$-invariant,
\newline
(iii) $p$ is the type at $+\infty$ or the type at $-\infty$ (so definable and $G$-invariant, hence $f$-generic).
\end{Lemma}
\begin{proof}  As remarked earlier we can and will identify $G$ with an open interval on which the group operation is 
continuous, and write $G$ additively (it is commutative). We know (or it is clear) that the types at $+\infty$ and 
$-\infty$ are $G$ invariant hence have bounded orbit. So it suffices to prove that any other type 
$q(x)\in S_{G}({\bar M})$ has unbounded $G$-orbit. This is really obvious but we go through details. So $q$ defines a cut in $G$ with nonempty left hand side $L$ and right hand side $R$. Let $a\in L$, $b\in R$ and $c = b-a > 0$. By compactness and saturation we can clearly find an increasing sequence $(d_{i}:i<\bar\kappa)$ in  $G$, such that $i < j$ implies $(d_{j}-d_{i}) \geq c$. Hence $\{d_{i}+ q: i< \bar\kappa\}$ witnesses that $q$ has unbounded orbit.

\end{proof}

%In the last part of this section we will give an explicit description of types with bounded orbit when $T$ is $o$-minimal, yielding in particular Conjecture ?? in the $o$-minimal case. We already know (in the $o$-minimal case) that if $G$ has some bounded orbit then we have the exact sequence $$1 \to W \to G\to C \to 1$$ where $W$ is solvable with no definably compact parts and $C$ is definably compact. So the general idea aim is to classify the types with bounded orbit in each of the cases $G = W$ and $G = C$, and then put them together.  First putting stuff together. This is related to Lemmas 4.1 and 4.2, and one could find common generalizations.

%Work in an arbitrary theory $T$ but the existence of Skolem functions will be built in to the additional assumptions.  Fix $\emptyset$-definable short exact sequence $1 \to L \to G \to_{f} H \to 1$ of groups. Let us assume the existence of an $\emptyset$-definable section $s: H \to G$ of $f$, but $s$ is NOT assumed to be a homomorphism. 

\vspace{2mm}
\noindent
%{\bf Construction A.}
%\newline
%Let $q(y)\in S_{H}({\bar M}$, and $r(x)\in S_{L}({\bar M})$.  Assume that $r$ is $Aut({\bar M}/A)$-invariant for some small $A$. Define a type $p(x)\in S_{G}({\bar M})$, by: given formula $\phi(x,c)$ over ${\bar M}$, $\phi(x,c)\in p(x)$ if for  $b$ realizing $q|(A,c)$, $\phi(x.s(b),c) \in r(x)$????  Note this is well defined, namely whether or not
%$\phi(x.s(b),c) \in r(x)$ depends only on $tp(b,c/A)$ by $A$-invariance of $r$.

%Then $p(x)$ is a complete global type of $G$ which we will write as  $r\otimes_{s}q$  (nonforking product across $s$).

%\begin{Lemma} If $q(y)$ and $r(x)$ both have bounded $H({\bar M})$ and  $L({\bar M})$ orbits respectively then $p = r\otimes_{s}q$ also has bounded $G({\bar M})$.
%\end{Lemma}

%\vspace{2mm}
%\noindent
%{\bf Construction B.}  Now let $p(x)$ be a global type of $G$, Let $q(y) = f(p) = tp(f(a))/{\bar M})$, where $a$ realizes $p$ in some elementary extension of ${\bar M}$. With this notation, let $r(x) = k(p)$  ($k$ for kernel)) be defined to be $tp(a.s(f(a))^{-1}/{\bar M})$.  
%Note that this agrees with the construction in 4.1. Namely a definable subset $Y$ of $L$ is in $r$ if and only if $Y\cdot S \in p$  where $S$ is the image of the section $s$. In any case $q(y)\in S_{G}({\bar M})$ and $r(x)\in S_{L}({\bar M})$.  With this notation:

%\begin{Lemma} (i) Suppose $p$ has bounded $G$ orbit. Then $f(p), k(p)$ have bounded $H,L$-orbits respectively.
%\newline
%(ii) Suppose that $r = k(p)$ is $Aut{\bar M}/A)$-invariant for some small $A$. Then $p = r\otimes_{s}q$  (where $q = f(p)$). 

%\end{Lemma}

\end{document}